\documentclass[12pt,reqno]{amsart}

\usepackage{amsmath,amsthm,amssymb}

\usepackage[all]{xy}

\newtheorem{thm}[subsection]{Theorem}
\newtheorem{prop}[subsection]{Proposition}
\newtheorem{lem}[subsection]{Lemma}

\theoremstyle{definition}
\newtheorem{Def}[subsection]{Definition}
\newtheorem{rem}[subsection]{Remark}

\newcommand{\PP}{{\mathbb P}}
\newcommand{\OOO}{{\mathcal O}}

\numberwithin{equation}{section}

\begin{document}

\author[I. Biswas]{Indranil Biswas}

\address{Department of Mathematics, Shiv Nadar University, NH91, Tehsil Dadri,
Greater Noida, Uttar Pradesh 201314, India}

\email{indranil.biswas@snu.edu.in, indranil29@gmail.com}

\author[F. Laytimi]{Fatima Laytimi}

\address {Math\'ematiques - b\^{a}t. M2, Universit\'e Lille 1,
F-59655 Villeneuve d'Ascq Cedex, France}

\email {fatima.laytimi@univ-lille.fr}

\author[D.S. Nagaraj]{D. S. Nagaraj}

\address{Indian Institute of science education and research, Tirupati, Srinivasapuram-Jangalapalli Village, 
Panguru (G.P) Yerpedu Mandal, Tirupati - 517619, Chittoor Dist., Andhra Pradesh India} 
\email{dsn@labs.iisertirupati.ac.in}

\author[W. Nahm]{Werner Nahm}
\address{Dublin Institute for Advanced Studies, 10 Burlington Road,
Dublin 4, Ireland}
\email{wnahm@stp.dias.ie}



\title{On anti-ample vector bundles and nef and big vector bundles}

\date{}

\begin{abstract}
We prove that the direct image of an anti-ample vector bundle is anti-ample under any finite flat morphism of 
non-singular projective varieties. In the second part we prove some properties of big and nef vector bundles. 
In particular it is shown that the tensor product of a nef vector bundle with a nef and big vector bundle is again 
nef and big. This generalizes a result of Schneider.
\end{abstract}

\maketitle

\section{Introduction} \setcounter{page}{1}

All the varieties considered in this article are defined over the field $\mathbb C$ of complex numbers.

Let $\pi\,:\, X \,\longrightarrow\, Y$ be a finite flat morphism between non-singular projective varieties.
If $E$ is an ample vector bundle on $X,$ in general the direct image vector bundle 
$\pi_* E$ need not be ample. However we prove that the dual of an ample bundle behaves well under direct image.
More precisely, the following is proved.

\begin{thm}\label{main1} Let $\pi\,:\, X \,\longrightarrow\, Y$ be a finite surjective morphism between non-singular 
projective varieties. 
If $E$ is an anti-ample (respectively, anti-nef) vector bundle on $X,$ then the direct image bundle 
$\pi_*E$ is an anti-ample (respectively, anti-nef) vector bundle on $Y$.
\end{thm}

It may be mentioned that in general direct image of a non anti-ample bundle may be anti-ample (see Remark \ref{Rem1}).

In the second part we establish some basic properties of big vector bundles and also nef and big vector bundles.

\begin{thm}\label{main2} 
If $E$ and $F$ are big vector bundles on a projective variety $X,$ then the tensor product $E\otimes F$ is also big.
\end{thm}

\begin{thm}\label{main3}
If $E$ is a big vector bundle on a projective variety $X,$ then the symmetric power $S^m(E)$ is big for all $m\,>\,0.$
\end{thm}

The following result is a generalization of a theorem of M. Schneider \cite{Sch} (see Remark \ref{Rem3}).

\begin{thm}\label{main4} 
If $E$ and $F$ are nef vector bundles on a projective variety $X,$ and one of them is also big, then $E\otimes F$ is 
nef and big.
\end{thm}

\section{Direct image of anti-ample bundles}

For standard notation and facts about vector bundles used here we refer to \cite{La1}, \cite{La2}. 

\begin{Def} 
A vector bundle $E$ on a projective variety over ${\mathbb C}$ is said to be anti-ample (respectively,
anti-nef) if the dual vector bundle $E^*$ is ample (respectively, nef).
\end{Def}

\begin{lem}\label{lem01}
Let $X$ and $Y$ be non-singular irreducible projective curves and $\pi \,:\, X \,\longrightarrow\, Y$
a nonconstant morphism. Then for any anti-nef bundle $E$ on $X$,
the direct image $\pi_*E$ is an anti-nef vector bundle.
\end{lem}
 
\begin{proof}
By the dual version of the criterion of nef-ness of a vector bundle (see, 
\cite[Proposition 16.1.18 (i)]{La2}) it is enough to prove the following: For any
nonconstant morphism $f\, :\, C\, \longrightarrow\, Y$ from any irreducible smooth projective curve $C$,
and any line subbundle $\phi\, :\, L \,\hookrightarrow\, f^*\pi_*(E)$ on $C$, the inequality
\begin{equation}\label{e1}
\text{degree}(L) \,\, \leq\,\, 0
\end{equation}
holds.

There is a non-singular projective curve $\widetilde{C}$ together with nonconstant morphisms
$$\widetilde{\pi}\,:\, \widetilde{C} \,\longrightarrow\, C\,\ \ \text{ and }\,\ \ \widetilde{f}\,:\,
\widetilde{C}\,\longrightarrow\, X$$ satisfying the condition
$f\circ \widetilde{\pi}\,=\, \pi\circ\widetilde{f}$; for example, take $\widetilde C$ to be the
normalization of any irreducible component of dimension one of the fiber product $C\times_Y X$.
For any line subbundle
\begin{equation}\label{el}
\phi\,\,:\,\, L\,\, \longrightarrow\,\, f^*\pi_*E,
\end{equation}
we have the line subbundle
$$
\widetilde{\pi}^*\phi\,\, :\,\, \widetilde{\pi}^*L \,\, \longrightarrow\,\, \widetilde{\pi}^*f^*\pi_* E
$$
of $\widetilde{\pi}^*f^*\pi_*E \, \longrightarrow\, \widetilde{C}$. Consider the fiber product
\begin{equation}\label{fb}
\begin{matrix}
C\times_Y X & \stackrel{p}{\longrightarrow} & X\\
\,\,\, \Big\downarrow q && \,\,\, \Big\downarrow\pi\\
C & \stackrel{f}{\longrightarrow} & Y
\end{matrix}
\end{equation}
We have $q_*(p^* E) \,=\, f^*(\pi_*E)$. Hence $\phi$ in \eqref{el} produces a homomorphism
\begin{equation}\label{el2}
\widehat{\phi}\,\, :\,\, q^*L \,\, \longrightarrow\,\, p^*E
\end{equation}
(see \cite[p.~110]{Ha}).
Consider the map $(\widetilde{\pi},\, \widetilde{f})\, :\, \widetilde{C} \,\longrightarrow\, C\times_Y X$.
Let
\begin{equation}\label{el3}
(\widetilde{\pi},\, \widetilde{f})^*\widehat{\phi}\,\, :\,\, (\widetilde{\pi},\, \widetilde{f})^*q^*L\,\,
\longrightarrow\,\, (\widetilde{\pi},\, \widetilde{f})^*p^* E
\end{equation}
be the pullback of $\widehat{\phi}$ (see \eqref{el2}). The vector bundle
$(\widetilde{\pi},\, \widetilde{f})^*p^*E$ is anti-nef because $E$ is anti-nef and
$p\circ (\widetilde{\pi},\, \widetilde{f})$ is a finite morphism. Since
$(\widetilde{\pi},\, \widetilde{f})^*p^*E$ is anti-nef, its line subbundle 
$(\widetilde{\pi},\, \widetilde{f})^*q^* L$ in \eqref{el3} satisfies the following condition:
$$
\text{degree}((\widetilde{\pi},\, \widetilde{f})^*q^*L)\,\, \leq\,\, 0.
$$
This implies that \eqref{e1} holds. This completes the proof.
\end{proof} 

\begin{proof}[{Proof of Theorem \ref{main1}}]
Let $\pi\,: \,X \,\longrightarrow\, Y$ be a finite surjective morphism of non-singular irreducible
projective varieties. Let $E$ be a vector bundle on $X$.

First assume that $E$ is anti-nef. To prove that $\pi_*E$ is anti-nef, it suffices to show the following:

If $f\, :\, C\, \longrightarrow\, Y$ is any nonconstant morphism from any irreducible smooth
projective curve $C$, then $f^*(\pi_*E)$ is anti-nef.

Let
\begin{equation}\label{ca}
\begin{matrix}
C\times_Y X & \stackrel{p}{\longrightarrow} & X\\
\,\,\, \Big\downarrow q && \,\,\, \Big\downarrow\pi\\
C & \stackrel{f}{\longrightarrow} & Y
\end{matrix}
\end{equation}
be the fiber product. Let $\varphi\,:\, Z \,\longrightarrow\, C\times_Y X$ be the normalization.
Note that we have
\begin{equation}\label{a1}
q_*p^*E \,\, \subset\,\, (q\circ\varphi)_* (\varphi^*p^*E)
\end{equation}
because the following diagram is commutative
$$
\begin{matrix}
Z & \stackrel{\varphi}{\longrightarrow}\ & C\times_Y X\\
\,\,\,\,\,\,\,\,\,\,\, \Big\downarrow q\circ\varphi && \,\, \Big\downarrow q\\
C & \stackrel{\rm Id}{\longrightarrow} & C
\end{matrix}
$$

The vector bundle $\varphi^*p^*E$ is anti-nef because $E$ is anti-nef and $p\circ\varphi$ is a finite map onto
its image. Therefore, from Lemma \ref{lem01} we conclude that $(q\circ\varphi)_* (\varphi^*p^*E)$ is anti-nef.
This implies that its subsheaf $q_*(p^*E)$ (see \eqref{a1}) is also anti-nef. But
$q_*(p^*E)\,=\, f^*(\pi_* E)$ because the diagram in \eqref{ca} is Cartesian. Hence
$f^*(\pi_*E)$ is anti-nef. This implies the $\pi_* E$ is anti-nef.

Now assume that $E$ is anti-ample.

Let $D$ be an ample divisor on $Y$. The $\mathbb Q$--twisted vector bundle
$E\langle \frac{1}{m} \pi^*(D)\rangle$ is again an anti-ample bundle on $X$ for all large integers $m$
(here we are using the notion of $\mathbb Q$--twisted bundle as in 
\cite[Ch.~ 6.2]{La2}). Hence we conclude that the $\mathbb Q$-twisted bundle 
$$\pi_*\left(E\langle \frac{1}{m} \pi^*(D)\rangle\right) \,\,=\,\, \pi_*(E)\langle \frac{1}{m} D\rangle $$ is anti-nef for
large $m$. Since the $\mathbb Q$--divisor $-\frac{1}{m} D$ is anti-ample, we conclude that the vector bundle 
$$\pi_*(E) \,=\, ((\pi_* E)\langle \frac{1}{m} D\rangle)\langle -\frac{1}{m} D\rangle$$ is anti-ample.
\end{proof}

\begin{rem}\label{Rem1}
Let $C$ be a non-singular curve of genus 1, and let $\pi: C \,\longrightarrow\, {\mathbb P}^1$ be
a morphism of degree two. If $L$ is any nontrivial line bundle on $C$ of degree $0$,
then it can be seen that $L$ is anti-nef but it is not anti-ample, while
$$\pi_*(L) \,=\, {\mathcal O}_{{\mathbb P}^1}(-1)^{\oplus 2}$$
is anti-ample.
\end{rem}

\section{big and nef vector bundles} 
 
The characterization 2.2.7 of big line bundle in \cite{La1} is equivalent to the following:
 
\begin{Def}\label{Def1} A line bundle $L$ on a projective variety $X$ is big if and only if there is
an ample line bundle $A$ on $X$ and a positive number $m$ such that
$$ H^0(X,\, L^m\otimes A^*)\,\neq\, 0.$$ 
\end{Def} 

A vector bundle $E$ is big if and only if $\OOO_{\PP(E)}(1)$ is big.

We need to recall a theorem from \cite{La1}.

\begin{thm}[{\cite[Theorem 2.2.16]{La1}}]\label{Rem2a}
A nef line bundle $L$ on a non-singular projective variety $X$ of dimension $n$ is big if and only if $c_1(L)^n \,>\,0,$
where $c_1(L)$ is the first Chern class of $L$ 

For a vector bundle this bigness criterion translates into the following:

A nef vector bundle
$E$ on a non-singular projective variety $X$ of dimension $n$ is big if and only if $(-1)^n s_n(E)
\,>\,0$, where $s_n(E)$ is the top Segre class of $E$.
\end{thm}

The following two propositions will be needed in the proofs of the remaining theorems.

\begin{prop}\label{main7} 
Let $E$ be a big vector bundle on $X$, and let $B$ be any line bundle on $X$. Then there is an integer $m\,\geq\, 1$ such that
$$H^0(X,\,S^m(E)\otimes B)\,\,\neq\,\, 0.$$
\end{prop}

\begin{proof}
Since $E$ is a big vector bundle on $X$, we deduce from the definition of bigness that there is an ample line bundle $A$
on $X$ such that 
$$H^0(X,\,S^j(E)\otimes A^*)\,\,\neq\,\, 0$$ for some integer $j \,\geq\, 1.$ Since $A$ is ample, there is an integer $l
\,\geq\, 1$ such that
$$H^0(X,\, A^{l}\otimes B)\,\,\neq\,\, 0.$$ If $s$ is a non-zero section of $S^j(E)\otimes A^*$, and $t$ is
a non-zero section of $A^{l} \otimes B$, then
$\widetilde{s}^l \otimes (\pi^*t)$ is a non-zero section of the line-bundle $\OOO_{\PP(E)}(jl)\otimes (\pi^*B),$ where
$\widetilde{s}$ is the section of
$\OOO_{\PP(E)}(j)\otimes \pi^*(A^*)$ corresponding to the section $s$ and $\pi \,:\, \PP(E) \,
\longrightarrow\, X$ is the natural projection map. The non-zero section
$\widetilde{s}^l \otimes \pi^*(t)$ of $\OOO_{\PP(E)}(jl)\otimes (\pi^*B)$ corresponds to a non-zero section
of $S^{jl}(E)\otimes B.$ This completes the proof.
\end{proof}

\begin{prop}\label{main8}
Let $E$ be a vector bundle on $X$, and let $A$ be any ample line bundle on $X$. The vector bundle $E$ is
big if and only if
$$H^0(X,\,S^mE\otimes A^*)\,\,\neq \,\,0$$
for some $m \,\geq \,1$. 
\end{prop}

\begin{proof}
The condition 
$$H^0(X,\,S^m(E)\otimes A^*)\,\,\neq\,\, 0$$
is equivalent to the condition
$$H^0(\PP(E) ,\, \OOO_{\PP(E)}(m)\otimes \pi^*(A^*))\,\,\neq\,\, 0,$$
where $\pi$ as before is the projection of $\PP(E)$ to $X$.
If $s$ is a non-zero section of $\OOO_{\PP(E)}(m)\otimes \pi^*(A^*)$ then $s^j$ is a non-zero section of 
$\OOO_{\PP(E)}(mj)\otimes (\pi^* A^*)^j$ for all integer $j \,\geq\, 1.$ On the other hand, $\OOO_{\PP(E)}(1)\otimes \pi^*(A^{l})$ 
is ample on $\PP(E)$ for some integer $l \,>\, 0$ (see \cite {Ha} Proposition 7.10(b)). 
In view of Definition \ref{Def1}, this implies that $E$ is a big vector bundle.

Conversely, if $E$ is big, setting $B\,=\, A^*$ in Proposition \ref{main7} we obtained the required non-vanishing
result in the statement of the proposition.
\end{proof} 

\begin{proof}[{\bf Proof of Theorem \ref{main2}}]
Let $A$ be an ample line bundle on $X$. Since $E$ and $F$ are assumed to be big, there exist positive integers $m$ and $n$
such that
$$H^0( S^{m} (E)\otimes A^*)\,\,\neq\,\, 0$$ and
$$H^0(X,\, S^{n} (F)\otimes A^*)\,\,\neq\,\, 0.$$
The theorem follows immediately from Proposition \ref{main8} if we prove that
\begin{equation}\label{env}
H^0(X,\, S^t (E\otimes F )\otimes (A^s)^*)\,\,\neq\,\, 0
\end{equation}
for some integers $t,\,s \,\geq\, 1.$

Now, \eqref{env} is a consequence of the following:

\begin{lem}\label{main6}
Let $L$ be a line bundle  and $E$ and $F$ are vector bundles. Let $m_1, \,m_2$ be positive integers  such that $S^{m_1}(E)\otimes L$ and $S^{m_2}(F)\otimes L$
have nonzero sections. Then $S^{m_1 m_2}(E\otimes F)\otimes L^{m_1+m_2}$ has a nonzero section.
\end{lem}
 
\begin{proof}
Let $s_1$ (respectively, $s_2$) be a non-zero section of $S^{m_1}(E)\otimes L$ (respectively, $S^{m_2}(F)\otimes L$). 
By the usual polarization argument, there are vectors $v_{1}\,\in \,E_{x}^*$ and $w_{1}\,\in\, L_{x}^*$
(respectively, $v_{2}\,\in \,F_{x}^*$ and $w_{2}\,\in\, L_{x}^*$)
such that $$\langle v_i^{\otimes m_i}\otimes w_i,\, s_i(x)\rangle\,\,\neq\,\, 0$$ for $i\,=\,1,\,2.$ There are the
standard natural maps
$$S^{m_1}(E)\otimes L\,\,\,\hookrightarrow\,\,\, E^{\otimes m_1}\otimes L$$
and
$$S^{m_2}(F)\otimes L\,\,\,\hookrightarrow\,\,\, F^{\otimes m_2}\otimes L,$$
and isomorphisms 
$$(E^{\otimes m_1}\otimes L)^{\otimes m_{2}} \,\,\,\simeq\,\,\, E^{\otimes m_1m_2}\otimes L^{m_{2}},$$
$$(F^{\otimes m_2}\otimes L)^{\otimes m_{1}} \,\,\,\simeq\,\,\, F^{\otimes m_1m_2}\otimes L^{m_{1}},$$
$$(E^{\otimes m_1m_2}\otimes L^{m_{2}})\otimes (F^{\otimes m_1m_2}\otimes L^{m_{1}})\,\simeq\,
(E\otimes F)^{\otimes m_1m_2}\otimes L^{m_{2}}\otimes L^{m_{1}},$$
and also there is the surjective map
$$(E\otimes F)^{\otimes m_1m_2}\otimes L^{m_{1}+m_{2}}\,\,\twoheadrightarrow\,\, S^{m_1m_2}(E\otimes F)\otimes L^{m_{1}+m_{2}}.$$
Applying these maps in sequence to $s_1,\, s_2$ we get a section $s$ of 
$$S^{m_1m_2}(E\otimes F)\otimes L^{m_1+m_2}.$$ It suffices to show that
\begin{equation}\label{es}
s\,\,\,\neq\,\,\, 0.
\end{equation}
Note that \eqref{es} holds because
$$\langle (v_1\otimes v_2)^{\otimes m_1m_2}\otimes w_1^{m_2}\otimes w_2^{m_1},\, s(x)\rangle \,\,= $$
$$\langle v_1^{\otimes m_1}\otimes w_1,\, s_1(x)\rangle^{m_2} \langle v_2^{\otimes m_2}\otimes w_2,\, s_2(x)\rangle^{m_1} \,\neq\, 0.$$
This completes the proof of the lemma.
\end{proof}

As noted before, Lemma \ref{main6} completes the proof of Theorem \ref{main2}.
\end{proof}

\begin{proof}[{\bf Proof of the Theorem \ref{main3}:}]
Since $E$ is big, there is an ample line bundle $A$ on $X$ and an integer $j \,>\, 0$, such that
$S^j(E)\otimes A^*$ has a non-zero section $s$ (see Proposition \ref{main8}). Then $s^m$ is a non-vanishing section of $S^{mj}(E)
\otimes (A^*)^m$. Since the natural map from $S^{mj}(E)\otimes (A^*)^m$ to $S^j(S^m(E))\otimes (A^*)^m$ is injective, it follows
that $S^j(S^m(E))\otimes (A^*)^m$ has a non-zero section. In view of Proposition \ref{main8}, this implies that $S^m(E)$ is big. 
\end{proof}

\begin{proof}[{\bf Proof of Theorem \ref{main4}}]
Assume that $E$ is big. Let $A$ be an ample line bundle on $X$ such that
\begin{equation}\label{j1}
H^0(X,\,S^m(E)\otimes A^*)\,\,\neq\,\, 0
\end{equation}
for some $m\,>\,0$. Let $$f\,\,:\,\,\widetilde{X}\,\,\longrightarrow\,\, X$$ be a finite ramified covering map such that
there is a line bundle $\widetilde A$ on $\widetilde X$ for which $f^*A\,=\,{\widetilde A}^{2m}$  {\cite[Theorem 4.1.10 (Bloch-Gieseker covering)]{La1}}.
Then \eqref{j1} implies
that $$H^0(X,\,S^m(f^*(E)\otimes \widetilde{A}^*)\otimes (\widetilde{A}^*)^m) \,\,\neq\,\, 0.$$
Thus the vector bundle $f^*(E)\otimes \widetilde{A}^*$ is big. On the other hand, since the vector
bundle $F$ in Theorem \ref{main4} is nef, and $\widetilde{A}$ is
ample, we conclude that $\widetilde{A}\otimes f^*F$ is ample, and hence it is big. Since $f^*(E)\otimes \widetilde{A}^*$ and
$\widetilde{A}\otimes f^*F$ are big, Theorem \ref{main2} says that
$$
f^*(E)\otimes \widetilde{A}^*\otimes\widetilde{A}\otimes f^*F\,\,=\,\, f^*(E\otimes F)
$$ is also big.

The tensor product of nef vector bundles is nef (see
\cite[Theorem 6.2.12]{La2}), and the pullback of a nef vector bundle --- under a surjective 
morphism --- is nef, we conclude that $f^*(E\otimes F)$ is nef. Hence
$f^*(E\otimes F)$ is nef and big. If the pullback $\Phi^*W$ of a vector bundle $W$ by a finite surjective morphism
$\Phi$ is nef, then $W$ is nef. Consequently, $E\otimes F$ is nef. Also, the pullback by  finite covers does not change the sign of
the Segre class, so that $E\otimes F$ is nef and big (see Theorem \ref{Rem2a}). 
\end{proof}

\begin{rem}\label{Rem3}
In \cite{Sch} it is shown that tensor product of a nef line bundle with a nef vector bundle is nef and big, provided one of
them is also big. Hence Theorem \ref{main4} is a generalization of Schneider's result.
\end{rem}

{\bf Acknowledgments}:  \  This work was supported in part by the Labex CEMPI  (ANR-11-LABX-0007-01).  We thank the referee for his
suggestions.

\pagebreak

{\bf  Data Availability Statements}

{\bf No data is used.}


\begin{thebibliography}{000}

\bibitem{Hart} R. Hartshorne, Ample vector bundles on curves, {\it Nagoya Math. Jour.} {\bf 43} (1971), 73--89.

\bibitem{Ha} R. Hartshorne, {\it Algebraic geometry}, Graduate Texts in Mathematics, No. 52.
Springer-Verlag, New York-Heidelberg, 1977.

\bibitem{Sch} M. Schneider, Some remarks on vanishing theorems for holomorphic vector bundles,
{\it Math. Zeit.}  {\bf 186} (1984), 135--142.

\bibitem{La1} R. Lazarsfeld, {\it Positivity in algebraic geometry. I}, Ergeb. Math. Grenzgeb. 48, Springer-Verlag, Berlin, 2004.

\bibitem{La2} R. Lazarsfeld, {\it Positivity in algebraic geometry. II}, Ergeb. Math. Grenzgeb. 49, Springer-Verlag, Berlin, 2004.

\end{thebibliography}
\end{document}